\documentclass{amsart}
\usepackage{amsfonts, amsbsy, amsmath, amssymb}

\input prepictex.tex
\input pictex.tex
\input postpictex.tex

\newtheorem{thm}{Theorem}[section]
\newtheorem{lem}[thm]{Lemma}

\newtheorem{prop}[thm]{Proposition}

\newtheorem{rmk}[thm]{Remark}

\newtheorem{thm-con}[thm]{Theorem-Conjecture}
\numberwithin{equation}{section}

\theoremstyle{definition}

\newcommand{\x}{{\tt X}}
\newcommand{\f}{\Bbb F}

\newcommand{\ar}[2]{\arrow <4pt> [0.3, 0.67] from #1 to #2}

\begin{document}

\title[Linearized Polynomial Equations]{From $r$-Linearized Polynomial Equations to $r^m$-Linearized Polynomial Equations}

\author[Neranga Fernando]{Neranga Fernando}
\address{Department of Mathematics,
Northeastern University, Boston, MA 02115}
\email{w.fernando@neu.edu}

\author[Xiang-dong Hou]{Xiang-dong Hou}
\address{Department of Mathematics and Statistics,
University of South Florida, Tampa, FL 33620}
\email{xhou@usf.edu}

\begin{abstract}
Let $r$ be a prime power and $q=r^m$. For $0\le i\le m-1$, let $f_i\in\f_r[\x]$ be $q$-linearized and $a_i\in\f_q$. Assume that $z\in\overline\f_r$ satisfies the equation $\sum_{i=0}^{m-1}a_if_i(z)^{r^i}=0$, where $\sum_{i=0}^{m-1}a_if_i^{r_i}\in\f_q[\x]$ is an $r$-linearized polynomial. It is shown that $z$ satisfies a $q$-linearized polynomial equation with coefficients in $\f_r$. This result provides an explanation for numerous permutation polynomials previously obtained through computer search. 
\end{abstract}



\maketitle

\section{Introduction}

Let $p$ be a prime and $\f_{q_1},\f_{q_2}\subset\overline\f_p$, where $\overline\f_p$ is the algebraic closure of $\f_p$. A $q_1$-linearized polynomial over $\f_{q_2}$ is a polynomial of the form 
\[
f=a_0\x^{q_1^0}+a_1\x^{q_1^1}+\cdots+a_n\x^{q_1^n}\in\f_{q_2}[\x].
\]
If $f\in\f_{q_2}[\x]$ is $q_1$-linearized and $g\in\f_{q_1}[\x]$ is $q_2$-linearized, then $f\circ g=g\circ f$. 

We are interested in the following question. Let $\f_r\subset \overline\f_p$ and $q=r^m$. Assume that $z\in\overline\f_p$ satisfies an equation 
\begin{equation}\label{1.1}
\sum_{i=0}^{m-1}a_if_i(z)^{r^i}=0,
\end{equation}
where $a_i\in\f_q$ and $f_i\in\f_r[\x]$ is $q$-linearized. Note that \eqref{1.1} is an $r$-linearized equation with coefficients in $\f_q$. Is it possible to derive from \eqref{1.1} a $q$-linearized equation with coefficients in $\f_r$? Such an equation indeed exists and will be given in Section 2; see \eqref{2.3}. The $q$-linearized equation \eqref{2.3} is concise in a determinant form. However, the expansion of the determinant invokes certain questions concerning permutations of $\Bbb Z_m$ and partitions of $m$; these related questions will also be discussed in Section 2. 

What is achieved is a transition from an $r$-linearized equation with coefficients in $\f_q$ to a $q$-linearized equation with coefficients in $\f_r$. We mention that for the applications in the present paper and other possible applications elsewhere, a $q$-linearized equation with coefficients in $\f_q$ would suffice. The advantage of a $q$-linearized polynomial $f\in\f_q[\x]$ over an $r$-linearized polynomial $g\in\f_q[\x]$ resides in a folklore in the study of finite fields. The conventional associate of $f$ is a polynomial in $\f_q[\x]$ from which information about the roots of $f$ can be easily extracted. On the other hand, the conventional associate of $g$ is a skew polynomial over $\f_q$ which is not as convenient to use as the counterpart of $f$. For more background of skew polynomials over finite fields, see \cite{Caruso-LeBorhne-arXiv1212.3582, Giesbrecht-JSC-1998}.

Let $f=\sum_{i=0}^na_i\x^{q^i}\in\overline\f_p[\x]$ be a $q$-linearized polynomial. The conventional associate of $f$ is the polynomial $\widetilde f=\sum_{i=0}^na_ix^i\in\overline\f_p[\x]$. The folklore mentioned above is the following theorem which can be derived from \cite[Theorems~ 3.62 and 3.65]{Lidl-Niederreiter-97}.

\begin{thm}\label{T1.1}
Let $f,g\in\f_q[\x]$ be $q$-linearized polynomials. Then $\text{\rm gcd}(f,g)$ is a $q$-linearized polynomial over $\f_q[\x]$ with $\widetilde{\text{\rm gcd}(f,g)}=\text{\rm gcd}(\widetilde f,\widetilde g)$. 
\end{thm}

In Section 3, we use the result of Section 2 to answer certain questions arising from a recent study of permutation polynomials \cite{Fernando-Hou-Lappano-FFA-2013,Hou-FFA-2012}. For each integer $n\ge 0$, let $g_{n,q}\in\f_p[\x]$ be the polynomial defined by the functional equation $\sum_{c\in\f_q}(\x+c)^n=g_{n,q}(\x^q-\x)$. The study in \cite{Fernando-Hou-Lappano-FFA-2013,Hou-FFA-2012} aims at the determination of the triples $(n,e;q)$ of positive integers for which $g_{n,q}$ is a permutation polynomial (PP) of $\f_{q^e}$. When $n$ is related to $q$ and $e$ in certain ways, the result of Section 2 combined with Theorem~\ref{T1.1} produces sufficient conditions for $g_{n,q}$ to be a PP of $\f_{q^e}$. These conditions explain numerous findings from the computer searches in \cite{Fernando-Hou-Lappano-FFA-2013}.


\section{$r$-Linearized and $r^m$-Linearized Equations}

Let $p$ be a prime, $\f_r\subset\overline\f_p$ and $q=r^m$. Let $R_q$ denote the set of all $q$-linearized polynomials over $\f_q$. $(R_q,+,\circ)$ is a commutative ring, where $+$ is the ordinary addition and $\circ$ is the composition. The mapping that sends $f\in R_q$ to its conventional associate $\widetilde f$ is a ring isomorphism from $R_q$ to $\f_q[\x]$. Throughout the paper, the $i$-fold composition of a polynomial $f$ is denoted by $f^{[i]}$ while $f^i$ stands for the $i$th power of $f$.

Assume that for $0\le i\le m-1$, $a_i\in\f_q$ and $f_i\in\f_r[\x]$ is $q$-linearized. Define
\begin{equation}\label{2.1}
M=\left[
\begin{matrix}
a_0f_0& a_1f_1&\cdots& a_{m-1}f_{m-1}\cr
a_{m-1}^rf_{m-1}\circ\x^q& a_0^rf_0&\cdots&a_{m-2}^rf_{m-2}\cr
\vdots&\vdots&&\vdots\cr
a_1^{r^{m-1}}f_1\circ\x^q& a_2^{r^{m-1}}f_2\circ\x^q&\cdots&a_0^{r^{m-1}}f_0
\end{matrix}\right]\in{\rm M}_{n\times n}(R_q).
\end{equation}

\begin{thm}\label{T2.1}
In the above notation, assume that $z\in\overline\f_p$ satisfies the equation
\begin{equation}\label{2.2}
\sum_{i=0}^{m-1}a_if_i^{r^i}(z)=0.
\end{equation}
Then we have
\begin{equation}\label{2.3}
(\det M)(z)=0,
\end{equation}
where $\det M$ is a $q$-linearized polynomial over $\f_r$.
\end{thm}

\begin{proof}
Raise the left side of \eqref{2.2} to the power of $r^j$, $0\le j\le m-1$, and express the results in a matrix form. We have
\begin{equation}\label{2.4}
\Bigl(\sum_{j=0}^{m-1}\left[
\begin{matrix}
a_jf_j^{r^j}\cr
a_{j-1}^rf_{j-1}^{r^j}\cr
\vdots\cr
a_0^{r^j}f_0^{r^j}\cr
a_{m-1}^{r^{j+1}}f_{m-1}^{r^j}\circ\x^q\cr
\vdots\cr
a_{j+1}^{r^{m-1}}f_{j+1}^{r^j}\circ\x^q
\end{matrix}\right]\Bigr)(z)=0.
\end{equation}
Label the rows and columns of $M$ from $0$ through $m-1$. Let $M_0$ be the submatrix $M$ with its $0$th column deleted, and, for $0\le i\le m-1$, let $M_{i,0}$ be the submatrix of $M_0$ with its $i$th row deleted. Put $D_i=(-1)^i\det M_{i,0}$, $0\le i\le m-1$. Then for each $1\le j\le m-1$, 
\begin{equation}\label{2.5}
\begin{split}
&\Bigl([D_0,\cdots,D_{m-1}]\circ \left[
\begin{matrix}
a_jf_j^{r^j}\cr
\vdots\cr
a_0^{r^j}f_0^{r^j}\cr
a_{m-1}^{r^{j+1}}f_{m-1}^{r^j}\circ\x^q\cr
\vdots\cr
a_{j+1}^{r^{m-1}}f_{j+1}^{r^j}\circ\x^q
\end{matrix}\right]\Bigr)(z)\cr
=\,&\Bigl([D_0,\cdots,D_{m-1}]\circ \left[
\begin{matrix}
a_jf_j\cr
\vdots\cr
a_0^{r^j}f_0\cr
a_{m-1}^{r^{j+1}}f_{m-1}\circ\x^q\cr
\vdots\cr
a_{j+1}^{r^{m-1}}f_{j+1}\circ\x^q
\end{matrix}\right]\Bigr)(z^{r^j})\cr
=\,&\Bigl(\det\left.\left[
\begin{matrix}
a_jf_j\cr
\vdots\cr
a_0^{r^j}f_0\cr
a_{m-1}^{r^{j+1}}f_{m-1}\circ\x^q\cr
\vdots\cr
a_{j+1}^{r^{m-1}}f_{j+1}\circ\x^q
\end{matrix}\ \right|\ M_0\ \right]\Bigr)(z^{r^j})=0.
\end{split}
\end{equation}
Combining \eqref{2.4} and \eqref{2.5} gives
\[
0=\Bigl([D_0,\cdots,D_{m-1}]\circ \left[
\begin{matrix}
a_0f_0\cr
a_{m-1}^rf_{m-1}\circ\x^q\cr
\vdots\cr
a_1^{r^{m-1}}f_1\circ\x^q
\end{matrix}\right]\Bigr)(z)=(\det M)(z).
\]
The claim that the coefficients of $\det M$ are all in $\f_r$ will be proved shortly; see \eqref{2.11} and \eqref{2.12}.
\end{proof}

For $i,j\in\{0,\dots,m-1\}$, define
\[
\delta(i,j)=
\begin{cases}
0&\text{if}\ i\le j,\cr
m&\text{if}\ i>j.
\end{cases}
\]
The $(i,j)$ entry of $M$ is 
\[
a_{-i+j}^{r^i}f_{-i+j}\circ\x^{r^{\delta(i,j)}},
\]
where the subscripts are taken modulo $m$. Let $\text{Sym}(\Bbb Z_m)$ denote the group of all permutations of $\Bbb Z_m$. We have
\begin{equation}\label{2.6}
\det M=\sum_{\sigma\in\text{Sym}(\Bbb Z_m)}(-1)^{\text{sgn}(\sigma)}\Bigl(\prod_{i=0}^{m-1}a_{\sigma(i)-i}^{r^i}f_{\sigma(i)-i}\Bigr)\circ\x^{r^{\sum_{i=0}^{m-1}\delta(i,\sigma(i))}}.
\end{equation}
Let $0^{\mu_0}1^{\mu_1}\cdots(m-1)^{\mu_{m-1}}$ denote the multiset consisting elements $0,\dots,m-1$ with respective multiplicities $\mu_0,\dots,\mu_{m-1}$. If there exists $\sigma\in\text{Sym}(\Bbb Z_m)$ such that $\{\sigma(i)-i:i\in\Bbb Z_m\}= 0^{\mu_0}1^{\mu_1}\cdots(m-1)^{\mu_{m-1}}$, then we must have $\sum_{i=0}^{m-1}i\mu_i\equiv 0\pmod m$. (In fact, the converse is also true; see Remark~\ref{R2.2}.) Let 
\begin{equation}\label{2.7}
\frak M=\Bigl\{(\mu_0,\mu_1,\dots,\mu_{m-1}):\mu_i\in\Bbb Z,\ \mu_i\ge 0,\ \sum_{i=0}^{m-1}\mu_i=m,\ \sum_{i=0}^{m-1}i\mu_i\equiv 0\pmod m\Bigr\},
\end{equation}
and for each $(\mu_0,\dots,\mu_{m-1})\in\frak M$, let
\begin{equation}\label{2.8}
\frak S_{(\mu_0,\dots,\mu_{m-1})}=\bigl\{\sigma\in\text{Sym}(\Bbb Z_m):\{\sigma(i)-i:i\in\Bbb Z_m\}=0^{\mu_0}\cdots(m-1)^{\mu_{m-1}}\bigr\}.
\end{equation}
Let $\sigma\in \frak S_{(\mu_0,\dots,\mu_{m-1})}$ and $i\in\{0,\dots,m-1\}$. Then $\delta(i,\sigma(i))+\sigma(i)-i$ is the integer in $\{0,\dots,m-1\}$ that is $\equiv\sigma(i)-i\pmod m$. Hence in $\Bbb Z$, we have
\begin{equation}\label{2.9}
\sum_{i=0}^{m-1}\delta(i,\sigma(i))=\sum_{i=0}^{m-1}\bigl[\delta(i,\sigma(i))+\sigma(i)-i\bigr]=\sum_{i=0}^{m-1}i\mu_i.
\end{equation}
Thus we can rewrite \eqref{2.6} as
\begin{equation}\label{2.10}
\det M=\sum_{(\mu_0,\dots,\mu_{m-1})\in\frak M}c_{\mu_0\dots\mu_{m-1}}f_0^{[\mu_0]}\circ\cdots\circ f_{m-1}^{[\mu_{m-1}]}\circ\x^{r^{0\mu_0+\cdots+(m-1)\mu_{m-1}}},
\end{equation}
where
\begin{equation}\label{2.11}
c_{\mu_0\dots\mu_{m-1}}=\sum_{\sigma\in\frak S_{(\mu_0,\dots,\mu_{m-1})}}(-1)^{\text{sgn}(\sigma)}\prod_{i=0}^{m-1}a_i^{\sum_{j,\,\sigma(j)-j\equiv i}r^j}.
\end{equation}
Let $\alpha\in\text{Sym}(\Bbb Z_m)$ be defined by $\alpha(i)=i-1$. We have $\alpha^{-1}\frak S_{(\mu_0,\dots,\mu_{m-1})}\alpha=\frak S_{(\mu_0,\dots,\mu_{m-1})}$. Moreover, for $\sigma\in \frak S_{(\mu_0,\dots,\mu_{m-1})}$,
\begin{equation}\label{2.12}
\Bigl(\prod_{i=0}^{m-1}a_i^{\sum_{j,\,\sigma(j)-j\equiv i}r^j}\Bigr)^r=\prod_{i=0}^{m-1}a_i^{\sum_{j,\,(\alpha^{-1}\sigma\alpha)(j)-j\equiv i}r^j}.
\end{equation}
\eqref{2.11} and \eqref{2.12} imply that $c_{\mu_0\cdots\mu_{m-1}}^r=c_{\mu_0\cdots\mu_{m-1}}$, i.e., $c_{\mu_0\cdots\mu_{m-1}}\in\f_r$.

\begin{rmk}\label{R2.2}\rm 
In fact, $\frak S_{(\mu_0,\dots,\mu_{m-1})}\ne\emptyset$ for all $(\mu_0,\dots,\mu_{m-1})\in\frak M$. This follows immediately from the following result by M. Hall.
\end{rmk}

\begin{thm}[\cite{Hall-PAMS-1952}]\label{T2.3}
Let $A$ be a finite abelian group and $f:A\to A$ a function. $f$ can be represented as a difference of two permutations of $A$ if and only if $\sum_{x\in A}f(x)=0$.
\end{thm}

We conclude this section with the explicit expansion of $\det M$ for $1\le m\le 5$. For simplicity, we write $f_{i_0}\circ\cdots\circ f_{i_{m-1}}\circ \x^{i_0+\cdots+i_{m-1}}$ as $f_{i_0\dots i_{m-1}}$.

$m=1$,
\[
\det M=a_0f_0.
\]
 
$m=2$,
\[
\det M=\text{N}_{r^2/r}(a_0)f_{00}-\text{N}_{r^2/r}(a_1)f_{11}.
\]

$m=3$,
\[
\det M=\text{N}_{r^3/r}(a_0)f_{000}+\text{N}_{r^3/r}(a_1)f_{111}+\text{N}_{r^3/r}(a_2)f_{222}-\text{Tr}_{r^3/r}(a_0a_1^ra_2^{r^2})f_{012}.
\]

$m=4$,
\[
\begin{split}
&\det M=\cr
&\text{N}_{r^4/r}(a_0)f_{0000}-\text{N}_{r^4/r}(a_1)f_{1111}+\text{N}_{r^4/r}(a_2)f_{2222}-\text{N}_{r^4/r}(a_3)f_{3333}\cr
&-\text{Tr}_{r^2/r}\bigl(\text{N}_{r^4/r^2}(a_0)\text{N}_{r^4/r^2}(a_2)^r\bigr)f_{0022}+\text{Tr}_{r^2/r}\bigl(\text{N}_{r^4/r^2}(a_1)\text{N}_{r^4/r^2}(a_3)^r\bigr)f_{1133}\cr
&-\text{Tr}_{r^4/r}(a_0^{1+r}a_1^{r^2}a_3^{r^2})f_{0013}+\text{Tr}_{r^4/r}(a_0a_1^{r+r^2}a_2^{r^3})f_{0112}\cr
&+\text{Tr}_{r^4/r}(a_0a_2^ra_3^{r^2+r^3})f_{0233}-\text{Tr}_{r^4/r}(a_1a_2^{r+r^2}a_3^{r^3})f_{1223}.
\end{split}
\]

$m=5$,
\[
\begin{split}
&\det M=\cr
&\text{N}_{r^5/r}(a_0)f_{00000}-\text{Tr}_{r^5/r}(a_0^{1+r+r^2}a_1^{r^3}a_4^{r^4})f_{00014}\cr
&-\text{Tr}_{r^5/r}(a_0^{1+r+r^3}a_2^{r^2}a_3^{r^4})f_{00023}+\text{Tr}_{r^5/r}(a_0^{1+r}a_1^{r^2+r^3}a_3^{r^4})f_{00113}\cr
&+\text{Tr}_{r^5/r}(a_0^{1+r^2}a_1^{r^3}a_2^{r+r^4})f_{00122}+\text{Tr}_{r^5/r}(a_0^{1+r}a_2^{r^2}a_4^{r^3+r^4})f_{00244}\cr
&+\text{Tr}_{r^5/r}(a_0^{1+r^2}a_3^{r+r^3}a_4^{r^4})f_{00334}-\text{Tr}_{r^5/r}(a_0a_1^{r+r^2+r^3}a_2^{r^4})f_{01112}\cr
&+\text{Tr}_{r^5/r}(a_0a_1^{r+r^3}a_4^{r^2+r^4})f_{01144}\cr
&+\text{Tr}_{r^5/r}(-a_0a_1^ra_2^{r^2}a_3^{r^3}a_4^{r^4}+a_0a_1^{r^2}a_2^{r^4}a_3^{r}a_4^{r^3}-a_0a_1^{r^3}a_2^{r}a_3^{r^4}a_4^{r^2})f_{01234}\cr
&-\text{Tr}_{r^5/r}(a_0a_1^{r^2}a_3^{r+r^3+r^4})f_{01333}-\text{Tr}_{r^5/r}(a_0a_2^{r+r^2+r^4}a_4^{r^3})f_{02224}\cr
&+\text{Tr}_{r^5/r}(a_0a_2^{r+r^2}a_3^{r^3+r^4})f_{02233}-\text{Tr}_{r^5/r}(a_0a_3^{r}a_4^{r^2+r^3+r^4})f_{03444}\cr
&+\text{N}_{r^5/r}(a_1)f_{11111}-\text{Tr}_{r^5/r}(a_1^{1+r+r^3}a_3^{r^2}a_4^{r^4})f_{11134}\cr
&+\text{Tr}_{r^5/r}(a_1^{1+r}a_2^{r^2+r^3}a_4^{r^4})f_{11224}+\text{Tr}_{r^5/r}(a_1^{1+r^2}a_2^{r^3}a_3^{r+r^4})f_{11233}\cr
&-\text{Tr}_{r^5/r}(a_1a_2^{r+r^2+r^3}a_3^{r^4})f_{12223}-\text{Tr}_{r^5/r}(a_1a_2^{r^2}a_4^{r+r^3+r^4})f_{12444}\cr
&+\text{Tr}_{r^5/r}(a_1a_3^{r+r^2}a_4^{r^3+r^4})f_{13344}+\text{N}_{r^5/r}(a_2)f_{22222}\cr
&+\text{Tr}_{r^5/r}(a_2^{1+r^2}a_3^{r^3}a_4^{r+r^4})f_{22344}-\text{Tr}_{r^5/r}(a_2a_3^{r+r^2+r^3}a_4^{r^4})f_{23334}\cr
&+\text{N}_{r^5/r}(a_3)f_{33333}+\text{N}_{r^5/r}(a_4)f_{44444}.
\end{split}
\]


\section{Applications to Permutation Polynomials}

\subsection{A criterion}\

\begin{prop}\label{P3.1}
Let $m$ and $e$ be positive integers, $r$ a prime power and $q=r^m$. Define $S_e=\x^{q^0}+\x^{q^1}+\cdots+\x^{q^{e-1}}$. A polynomial $f\in\f_{q^e}[\x]$ is a PP of $\f_{q^e}$ if the following conditions are all satisfied.
\begin{itemize}
  \item [(i)] There exists a PP $\bar f\in\f_q[\x]$ of $\f_q$ such that the diagram
\[
\beginpicture
\setcoordinatesystem units <3mm,3mm> point at 0 0

\ar{1 0}{5 0}
\ar{1 6}{5 6}
\ar{0 5}{0 1}
\ar{6 5}{6 1}
\put {$\f_{q^e}$} at 0 6
\put {$\kern1mm\f_{q^e}$} at 6 6
\put {$\f_q$} at 0 0
\put {$\f_q$} at 6 0
\put {$\scriptstyle f$} [b] at 3 6.2 
\put {$\scriptstyle \bar f$} [t] at 3 -0.2 
\put {$\scriptstyle S_e$} [r] at -0.2 3
\put {$\scriptstyle S_e$} [l] at 6.2 3
\endpicture
\]  
commutes.

\item[(ii)] For each $c\in\f_q$, there exist $q$-linearized polynomials $f_{c,i}\in\f_r[\x]$ and $a_{c,i}\in\f_q$, $0\le i\le m-1$, and $b_c\in\f_{q^e}$ such that
\begin{equation}\label{3.1}
f(x)=f_c(x)+b_c\quad\text{for all}\ x\in S_e^{-1}(c),
\end{equation}
where 
\begin{equation}\label{3.2}
f_c=\sum_{i=0}^{m-1}a_{c,i}f_{c,i}^{r^i}.
\end{equation}

\item[(iii)] For each $c\in\f_q$,
\begin{equation}\label{3.3}
\text{\rm gcd}\bigl(\det A_c,\ (\x^e-1)/(\x-1)\bigr)=1,
\end{equation}
where
\begin{equation}\label{3.4}
A_c=\left[
\begin{matrix}
a_{c,0}\widetilde f_{c,0} & a_{c,1}\widetilde f_{c,1} &\cdots& a_{c,m-1}\widetilde f_{c,m-1}\cr
a_{c,m-1}^r\widetilde f_{c,m-1}\x & a_{c,0}^r\widetilde f_{c,0} &\cdots& a_{c,m-2}^r\widetilde f_{c,m-2}\cr
\vdots&\vdots&&\vdots\cr
a_{c,1}^{r^{m-1}}\widetilde f_{c,1}\x & a_{c,2}^{r^{m-1}}\widetilde f_{c,2}\x &\cdots& a_{c,0}^{r^{m-1}}\widetilde f_{c,0}
\end{matrix}\right],
\end{equation}
\end{itemize}
and $\widetilde{(\ )}$ denotes the conventional associate of a $q$-linearized polynomial over $\f_q$.
\end{prop}

\begin{proof}
By \cite[Lemma~1.2]{Akbary-Ghioca-Wang-FFA-2011}, it suffices to show that for every $c\in\f_q$, $f_c$ is 1-1 on $S_e^{-1}(c)$. Since $f_c$ is a linearized polynomial, it suffices to show that $0$ is the only common root of $f_c$ and $S_e$. By Theorem~\ref{T2.1}, a root of $f_c$ is also a root of $\det M_c$, where
\begin{equation}\label{3.5}
M_c=\left[
\begin{matrix}
a_{c,0}f_{c,0} & a_{c,1}f_{c,1} &\cdots& a_{c,m-1}f_{c,m-1}\cr
a_{c,m-1}^rf_{c,m-1}\circ\x^q & a_{c,0}^rf_{c,0} &\cdots& a_{c,m-2}^rf_{c,m-2}\cr
\vdots&\vdots&&\vdots\cr
a_{c,1}^{r^{m-1}}f_{c,1}\circ\x^q & a_{c,2}^{r^{m-1}}f_{c,2}\circ\x^q &\cdots& a_{c,0}^{r^{m-1}}f_{c,0}
\end{matrix}\right].
\end{equation}  
Therefore, it suffices to show that $\text{gcd}(\det M_c,\,S_e)=\x$. We have
\[
\text{gcd}(\widetilde{\det M_c},\ \widetilde S_e)=\text{gcd}\bigl(\det A_c,\ (\x^e-1)/(\x-1)\bigr)=1,
\]
and by Theorem~\ref{T1.1}, $\text{gcd}(\det M_c,\,S_e)=\x$.
\end{proof}

\subsection{The polynomial $g_{n,q}$}\

Let $p=\text{char}\,\f_q$. For each integer $n\ge 0$, there is a polynomial $g_{n,q}\in\f_p[\x]$ defined by the functional equation
\begin{equation}\label{3.6}
\sum_{c\in\f_q}(\x+c)^n=g_{n,q}(\x^q-\x).
\end{equation}
The polynomial $g_{n,q}$ was introduced in \cite{Hou-JCTA-2011}, and its permutation property was studied in \cite{Fernando-Hou-Lappano-FFA-2013,Hou-FFA-2012}. The objective is to determine the triples $(n,e;q)$ of positive integers for which $g_{n,q}$ is a PP of $\f_{q^e}$, and this question, as a whole, appears to be difficult. For each integer $a\ge 0$, define $S_a=\x+\x^q+\cdots+\x^{q^{a-1}}$. The polynomial $g_{n,q}$ is related to $S_a$ by the following lemma.

\begin{lem}\label{L3.2}
\begin{itemize}
  \item [(i)] {\rm (\cite[Eq.~(3.5)]{Fernando-Hou-Lappano-FFA-2013})} For integers $n\ge 0$ and $a\ge b\ge 0$, we have
\begin{equation}\label{3.7}
g_{n+q^a,q}-g_{n+q^b,q}=(S_a-S_b)\cdot g_{n,q}.
\end{equation}

\item[(ii)] {\rm (\cite[Lemma~2.2]{Fernando-Hou-Lappano-DM-2014})} Let $n=1+q^{a_1}+\cdots+q^{a_{q+t}}$, where $-1\le t\le q-4$ and $a_1,\dots,a_{q+t}\ge 0$. Then
\begin{equation}\label{3.8}
g_{n,q}=-\sum_{1\le i_1<\cdots<i_{t+2}\le q+t}S_{a_{i_1}}\cdots S_{a_{i_{t+2}}}.
\end{equation}
\end{itemize}
\end{lem}

\subsection{Applications to $g_{n,q}$}\

We call a triple of positive integers $(n,e;q)$ {\em desirable} if $g_{n,q}$ is a PP of $\f_{q^e}$. Computer searches for desirable triples with small values of $q$ and $e$ were conducted in \cite{Fernando-Hou-Lappano-FFA-2013,Hou-FFA-2012}; the results for $q=4$ and $e\le 6$ were given in \cite[Table~3]{Fernando-Hou-Lappano-FFA-2013}. Proposition~\ref{P3.1} provides an explanation for several entries of \cite[Table~3]{Fernando-Hou-Lappano-FFA-2013}; in each case, a new class of PPs is discovered.

\begin{prop}\label{P3.3}
Let $q=4$ and $n=1+q^a+q^b+q^e+q^{e+k}$, where $a,b,e,$ and $k$ are positive integers. Then 
\begin{equation}\label{3.9}
g_{n,q}\equiv S_aS_b+(S_a+S_b+S_e)S_k+S_e^2\pmod{\x^{q^e}-\x}.
\end{equation}
If $\text{\rm gcd}(e,2k)=1$ and $a=k$ or $b=k$, then $g_{n,q}$ is a PP of $\f_{q^e}$.
\end{prop}

\begin{proof}
Eq.~\eqref{3.9} follows from Lemma~\ref{L3.2} easily. It remains to prove the second claim. Since $a=k$ or $b=k$, \eqref{3.9} gives
\[
g_{n,q}\equiv S_k^2+S_e^2+S_eS_k\pmod{\x^{q^e}-\x}.
\]
We show that conditions (i) -- (iii) in Proposition~\ref{P3.1} are satisfied with $r=2$, $m=2$, $q=4$ and $f=S_k^2+S_e^2+S_eS_k$. Condition (i) is satisfied with $\bar f=\x^2$. For each $c\in\f_q$ and $x\in S_e^{-1}(c)$,
\[
f(x)=cS_k(x)+S_k^2(x)+c^2.
\]
Hence (ii) is satisfied with
\[
f_{c,0}=f_{c,1}=S_k,\ a_{c,0}=c,\ a_{c,1}=1,\ b_c=c^2.
\]
We have 
\[
\det A_c=\left|\begin{matrix} c(1+\x+\cdots+\x^{k-1})&1+\x+\cdots+\x^{k-1}\cr \x(1+\x+\cdots+\x^{k-1})& c^2(1+\x+\cdots+\x^{k-1})\end{matrix}\right|=\frac{(\x^{2k}+1)(\x+c^3)}{(\x+1)^2}.
\]
Since $\text{gcd}(e,2k)=1$, $\text{gcd}(\det A_c,\ (\x^e-1)/(\x-1))=1$, and hence (iii) is also satisfied.
\end{proof}

The proofs of the next three propositions are similar to that of Proposition~\ref{P3.3}. For these proofs, we only give $f$, $f_{c,i}$, $a_{c,i}$, $b_c$ and $\det A_c$ and leave the details for the reader.

\begin{prop}\label{P3.4}
Let $q=4$ and $n=1+3q^a+q^e+2q^{e+a}$, where $e$ and $a$ are positive integers. Then
\[
g_{n,q}\equiv\x^{q^a}+ S_e+S_a^2S_e^2+S_aS_e^3\pmod{\x^{q^e}-\x}.
\]
If $2\mid e$ and $\text{\rm gcd}(e,2a+1)=1$, then $g_{n,q}$ is a PP of $\f_{q^e}$.
\end{prop}  

\begin{proof}
$f=\x^{q^a}+ S_e+S_a^2S_e^2+S_aS_e^3$, $\bar f=\x$, $f_{c,0}=\x^{q^a}+c^3S_a$, $f_{c,1}=S_a$, $a_{c,0}=1$, $a_{c,1}=c^2$, $b_c=c$, 
\[
\det A_c=\x^{2a}+c^3\frac{\x^{2a}+1}{\x+1}.
\]
\end{proof} 

\begin{prop}\label{P3.5}
Let $q=4$ and $n=1+2q^1+2q^{e-1}+2q^{e+1}$, where $e>1$ is an integer. Then
\[
g_{n,q}\equiv S_2+\x^2 S_e^2+S_{e-1}^2S_e^2\pmod{\x^{q^e}-\x}.
\]
If $e$ is odd, then $g_{n,q}$ is a PP of $\f_{q^e}$.
\end{prop}  

\begin{proof}
$f=S_2+\x^2 S_e^2+S_{e-1}^2S_e^2$, $\bar f=\x$, $f_{c,0}=S_2$, $f_{c,1}=\x+S_{e-1}$, $a_{c,0}=1$, $a_{c,1}=c^2$, $b_c=0$, 
\[
\det A_c=\frac1{(1+\x)^2}(1+c^3\x^3+\x^4+c^3\x^{2e-1}).
\]
\end{proof} 

\begin{prop}\label{P3.6}
Let $q=4$ and $n=1+2q^1+q^3+q^e+2q^{e+1}$, where $e$ is a positive integer. Then
\[
g_{n,q}\equiv \x^{q^2}+S_e+\x^2S_e^2+S_3S_e^3\pmod{\x^{q^e}-\x}.
\]
If $2\mid e$ but $3\nmid e$, then $g_{n,q}$ is a PP of $\f_{q^e}$.
\end{prop}  

\begin{proof}
$f=\x^{q^2}+S_e+\x^2S_e^2+S_3S_e^3$, $\bar f=\x$, $f_{c,0}=\x^{q^2}+c^3S_3$, $f_{c,1}=\x$, $a_{c,0}=1$, $a_{c,1}=c^2$, $b_c=c$, $\det A_c=\x^4+c^3(1+\x+\x^2+\x^4)$.
\end{proof} 

\begin{prop}\label{P3.7}
Let $q=4$ and $n=1+2q^2+q^4+q^e+2q^{e+2}$, where $e$ is a positive integer. Then
\[
g_{n,q}\equiv \x^{q^3}+S_e+S_2^2S_e^2+S_4S_e^3 \pmod{\x^{q^e}-\x}.
\]
If $2\mid e$ but $5\nmid e$, then $g_{n,q}$ is a PP of $\f_{q^e}$.
\end{prop}  

\begin{proof}
$f=\x^{q^3}+S_e+S_2^2S_e^2+S_4S_e^3$, $\bar f=\x$, $f_{c,0}=\x^{q^3}+c^3S_4$, $f_{c,1}=S_2$, $a_{c,0}=1$, $a_{c,1}=c^2$, $b_c=c$, $\det A_c=\x^6+c^3(1+\x+\x^2+\x^3+\x^4+\x^6)$.
\end{proof} 

Examples of Propositions~\ref{P3.3} -- \ref{P3.7} provide explanations for several entries in \cite[Table 3]{Fernando-Hou-Lappano-FFA-2013}; an update of the table is included in the apendix of the present paper.

By Proposition~\ref{P3.3}:
\[
\begin{split}
&q=4,\ e=5,\ n=17429=1+q^1+q^2+q^5+q^7;\\
&q=4,\ e=5,\ n=17489=1+q^2+q^3+q^5+q^7;\\
&q=4,\ e=5,\ n=17681=1+q^2+q^4+q^5+q^7.
\end{split}
\]

By Proposition~\ref{P3.4}:
\[
\begin{split}
&q=4,\ e=4,\ n=2317=1+3q^1+q^4+2q^5;\\
&q=4,\ e=6,\ n=135217=1+3q^2+q^6+2q^8.
\end{split}
\]

By Proposition~\ref{P3.5}:
\[
q=4,\ e=5,\ n=8713=1+2q^1+2q^4+2q^6.
\]

By Proposition~\ref{P3.6}:
\[
q=4,\ e=4,\ n=2377=1+2q^1+q^3+q^4+2q^5.
\]

By Proposition~\ref{P3.7}:
\[
q=4,\ e=6,\ n=135457=1+2q^2+q^4+q^6+2q^8.
\]

\begin{rmk}\label{R3.8}\rm
The statements in the above propositions can be made more general when they are not restricted to the class $g_{n,q}$. For example, Propositions~\ref{P3.5} and \ref{P3.6} allow the following generalizations whose proofs require no additional work.

\begin{itemize}
  \item A generalization of Proposition~\ref{P3.5}: Let $q=4$ and $f=S_a+\x^2S_e^2+S_b^2S_e^2$, where $a$, $b$, and $e$ are positive integers. Then $f$ is a PP of $\f_{q^e}$ if $2\mid(a+b)$, $\text{gcd}(e,a)=1$, and 
\[
\text{gcd}\Bigl(\frac{\x^{2a}+1+\x^{2b+1}+\x^3}{(\x+1)^2},\, \frac{\x^e+1}{\x+1}\Bigr)=1.
\]
In this case, $\bar f=\x$, $f_{c,0}=S_a$, $f_{c,1}=\x+S_b$, $a_{c,0}=1$, $a_{c,1}=c^2$, $b_c=0$, and 
\[
\det A_c=\frac 1{(1+\x)^2}\bigl[\x^{2a}+1+c^3(\x^{2b+1}+\x^3)\bigr].
\]

\item A generalization of Proposition~\ref{P3.6}: Let $q=4$ and $f=S_a+S_b+S_e+\x^2S_e^2+S_bS_e^3$, where $a$, $b$, and $e$ are positive integers. Then $f$ is a PP of $\f_{q^e}$ if $2\mid(a+e)$, $\text{gcd}(e,a-b)=1$, and 
\[
\text{gcd}\Bigl(\frac{\x^{2a}+\x^3+\x+1}{(\x+1)^2},\ \frac{\x^e+1}{\x+1}\Bigr)=1.
\]
In this case, $\bar f=\x$, $f_{c,0}=S_a+(1+c^3)S_b$, $f_{c,1}=\x$, $a_{c,0}=1$, $a_{c,1}=c^2$, $b_c=c$, and 
\[
\det A_c=\frac 1{(1+\x)^2}\bigl[\x^{2a}+\x^{2b}+c^3(1+\x+\x^3+\x^{2b})\bigr].
\]
\end{itemize}
\end{rmk}


%
%
\clearpage

\section*{Appendix}

The following is an update of \cite[Table~3]{Fernando-Hou-Lappano-FFA-2013} of all desirable triples $(n,e;4)$ with $e\le 6$ and $w_4(n)>4$, where $w_4(n)$ is the base $4$ weight of $n$. Only a few entries remain to be explained.

\begin{table}[!h]
\caption{Desirable triples $(n,e;4)$, $e\le 6$, $w_4(n)>4$}\label{Tb3}
\vspace{-5mm}
\[
\begin{tabular}{|c|r|l|c|}
\hline
$e$ & $n \hfil$ & base $4$ digits of $n$ & reference \\ \hline
\hline

2&59&{3,2,3}&  \cite[Theorem 5.9 (ii)]{Fernando-Hou-Lappano-FFA-2013}    \\  \hline
2&127&{3,3,3,1}&  \cite[Proposition 3.1]{Hou-FFA-2012}  \\  \hline
3&29&{1,3,1}&  \cite[Example 6.3]{Fernando-Hou-Lappano-FFA-2013}    \\  \hline
3&101&{1,1,2,1}&   \cite[Theorem 6.12 ]{Fernando-Hou-Lappano-FFA-2013}   \\  \hline
3&149&{1,1,1,2}&     \\  \hline
3&163&{3,0,2,2}&   \cite[Theorem 6.10]{Fernando-Hou-Lappano-FFA-2013}    \\  \hline
3&281&{1,2,1,0,1}&  \cite[Corollary 6.16]{Fernando-Hou-Lappano-FFA-2013}  \\  \hline
3&307&{3,0,3,0,1}&   \cite[Theorem 1.1]{Hou13}  \\  \hline
3&329&{1,2,0,1,1}&  \cite[Example 6.4]{Fernando-Hou-Lappano-FFA-2013}    \\  \hline
3&341&{1,1,1,1,1}&  \cite[Example 6.4 ]{Fernando-Hou-Lappano-FFA-2013}    \\  \hline
3&2047&{3,3,3,3,3,1}&  \cite[Proposition 3.1]{Hou-FFA-2012}    \\  \hline
4&281&{1,2,1,0,1}&    \cite[Theorem 6.12]{Fernando-Hou-Lappano-FFA-2013}   \\  \hline
4&307&{3,0,3,0,1}&     \\  \hline
4&401&{1,0,1,2,1}&   \cite[Theorem 6.12]{Fernando-Hou-Lappano-FFA-2013}   \\  \hline
4&547&{3,0,2,0,2}&   \cite[Theorem 6.10]{Fernando-Hou-Lappano-FFA-2013}  \\  \hline
4&779&{3,2,0,0,3}&   \cite[Theorem 6.6]{Fernando-Hou-Lappano-FFA-2013}    \\  \hline
4&787&{3,0,1,0,3}&   \cite[Theorem 6.8]{Fernando-Hou-Lappano-FFA-2013}    \\  \hline
4&817&{1,0,3,0,3}&     \\  \hline
4&899&{3,0,0,2,3}&   \cite[Theorem 6.6]{Fernando-Hou-Lappano-FFA-2013}    \\  \hline
4&1469&{1,3,3,2,1,1}&     \\  \hline
4&2201&{1,2,1,2,0,2}&     \\  \hline
4&2317&{1,3,0,0,1,2}&   Proposition~\ref{P3.4}  \\  \hline
4&2321&{1,0,1,0,1,2}&   \cite[Theorem 6.12]{Fernando-Hou-Lappano-FFA-2013}   \\  \hline
4&2377&{1,2,0,1,1,2}&   Proposition~\ref{P3.6} \\  \hline
4&2441&{1,2,0,2,1,2}&     \\  \hline
4&4387&{3,0,2,0,1,0,1}&     \\  \hline
4&32767&{3,3,3,3,3,3,3,1}&  \cite[Proposition 3.1]{Hou-FFA-2012}  \\  \hline
5&29&{1,3,1}&   \cite[ Example 6.3]{Fernando-Hou-Lappano-FFA-2013}  \\  \hline
5&1049&{1,2,1,0,0,1}&   \cite[Theorem 6.12]{Fernando-Hou-Lappano-FFA-2013}    \\  \hline
5&1061&{1,1,2,0,0,1}&  \cite[Theorem 6.12]{Fernando-Hou-Lappano-FFA-2013}    \\  \hline
5&1169&{1,0,1,2,0,1}&  \cite[Theorem 6.12]{Fernando-Hou-Lappano-FFA-2013}   \\  \hline
5&1289&{1,2,0,0,1,1}&  \cite[Theorem 6.12]{Fernando-Hou-Lappano-FFA-2013}    \\  \hline
5&1409&{1,0,0,2,1,1}&  \cite[Theorem 6.12]{Fernando-Hou-Lappano-FFA-2013}    \\  \hline
5&1541&{1,1,0,0,2,1}&   \cite[Theorem 6.12]{Fernando-Hou-Lappano-FFA-2013}  \\  \hline
5&1601&{1,0,0,1,2,1}&  \cite[Theorem 6.12]{Fernando-Hou-Lappano-FFA-2013}   \\  \hline
5&2083&{3,0,2,0,0,2}&  \cite[Theorem 6.10]{Fernando-Hou-Lappano-FFA-2013}     \\  \hline
5&2563&{3,0,0,0,2,2}&   \cite[Theorem 6.9]{Fernando-Hou-Lappano-FFA-2013}    \\  \hline

\end{tabular}
\]
\end{table}

\clearpage

\addtocounter{table}{-1}
\begin{table}[!h]
\caption{continued}
\vspace{-5mm}
\[
\begin{tabular}{|c|r|l|c|}
\hline
$e$ & $n \hfil$ & base $4$ digits of $n$ & reference \\ \hline
\hline

5&4229&{1,1,0,2,0,0,1}&   \cite[Theorem 6.12]{Fernando-Hou-Lappano-FFA-2013}    \\  \hline
5&4289&{1,0,0,3,0,0,1}&     \\  \hline
5&4387&{3,0,2,0,1,0,1}&     \\  \hline
5&5129&{1,2,0,0,0,1,1}&   \cite[ Example 6.4]{Fernando-Hou-Lappano-FFA-2013}   \\  \hline
5&5141&{1,1,1,0,0,1,1}&   \cite[Example 6.4 ]{Fernando-Hou-Lappano-FFA-2013}   \\  \hline
5&5189&{1,1,0,1,0,1,1}&   \cite[Example 6.4]{Fernando-Hou-Lappano-FFA-2013}    \\  \hline
5&5249&{1,0,0,2,0,1,1}&   \cite[Theorem 6.12]{Fernando-Hou-Lappano-FFA-2013}   \\  \hline
5&5381&{1,1,0,0,1,1,1}&   \cite[Example 6.4]{Fernando-Hou-Lappano-FFA-2013}    \\  \hline
5&8713&{1,2,0,0,2,0,2}&   Proposition~\ref{P3.5} \\  \hline
5&9281&{1,0,0,1,0,1,2}&   \cite[ Theorem 6.12 ]{Fernando-Hou-Lappano-FFA-2013}  \\  \hline
5&17429&{1,1,1,0,0,1,0,1}&   Proposition~\ref{P3.3} \\  \hline
5&17441&{1,0,2,0,0,1,0,1}&   \cite[Theorem 6.12]{Fernando-Hou-Lappano-FFA-2013}    \\  \hline
5&17489&{1,0,1,1,0,1,0,1}&   Proposition~\ref{P3.3} \\  \hline
5&17681&{1,0,1,0,1,1,0,1}&   Proposition~\ref{P3.3} \\  \hline
5&524287&{3,3,3,3,3,3,3,3,3,1}&  \cite[Proposition 3.1]{Hou-FFA-2012}    \\  \hline

6&4361&{1,2,0,0,1,0,1}&   \cite[Theorem 6.12]{Fernando-Hou-Lappano-FFA-2013}   \\  \hline
6&6161&{1,0,1,0,0,2,1}&    \cite[Theorem 6.12]{Fernando-Hou-Lappano-FFA-2013}    \\  \hline
6&6401&{1,0,0,0,1,2,1}&    \cite[Theorem 6.12]{Fernando-Hou-Lappano-FFA-2013}   \\  \hline
6&8227&{3,0,2,0,0,0,2}&   \cite[Theorem 6.10]{Fernando-Hou-Lappano-FFA-2013}   \\  \hline
6&8707&{3,0,0,0,2,0,2}&   \cite[Theorem 6.11]{Fernando-Hou-Lappano-FFA-2013}   \\  \hline
6&12299&{3,2,0,0,0,0,3}&   \cite[Theorem 6.6]{Fernando-Hou-Lappano-FFA-2013}   \\  \hline
6&12307&{3,0,1,0,0,0,3}&   \cite[Theorem 6.8 ]{Fernando-Hou-Lappano-FFA-2013}   \\  \hline
6&14339&{3,0,0,0,0,2,3}&   \cite[Theorem 6.6]{Fernando-Hou-Lappano-FFA-2013}    \\  \hline
6&37121&{1,0,0,0,1,0,1,2}&   \cite[Theorem 6.12]{Fernando-Hou-Lappano-FFA-2013}    \\  \hline
6&65801&{1,2,0,0,1,0,0,0,1}&   \cite[Corollary 6.16]{Fernando-Hou-Lappano-FFA-2013}    \\  \hline
6&65921&{1,0,0,2,1,0,0,0,1}&     \\  \hline
6&66307&{3,0,0,0,3,0,0,0,1}&    \cite[Theorem 1.1]{Hou13}  \\  \hline
6&135209&{1,2,2,0,0,0,1,0,2}&     \\  \hline
6&135217&{1,0,3,0,0,0,1,0,2}&   Proposition~\ref{P3.4} \\  \hline
6&135457&{1,0,2,0,1,0,1,0,2}&  Proposition~\ref{P3.7} \\  \hline
6&137249&{1,0,2,0,0,2,1,0,2}&     \\  \hline
6&8388607&{3,3,3,3,3,3,3,3,3,3,3,1}&  \cite[Proposition 3.1]{Hou-FFA-2012}    \\  \hline

\end{tabular}
\]
\end{table}

\end{document}